\documentclass{svjour3}                     
\smartqed  
\usepackage{amssymb}
\usepackage{graphicx}
\usepackage{lscape}
\usepackage{amscd}
\usepackage{cite}
\usepackage{yfonts}
\usepackage{amsmath}
\usepackage{amsfonts}
\usepackage{latexsym}
\newtheorem{prop}{Proposition}
\newtheorem{Cor}{Corollary}
\newtheorem{defin}{Definition}

\newtheorem{rem}{Remark}
\begin{document}

\title{Unified Bernoulli-Euler polynomials of Apostol type}


\author{
Hac\`{e}ne \textsc{Belbachir}, Yahia \textsc{Djemmada}, Slimane \textsc{Hadj-Brahim}
}


\institute{H. BELBACHIR \at
               USTHB, Faculty of Mathematics, RECITS Laboratory P. Box $32$,\\ El-Alia, $16111$ Bab-Ezzouar, Algiers, Algeria.\\
              \email{hacenebelbachir@gmail.com or hbelbachir@usthb.dz}
              \and
              Y. DJEMMADA (corresponding author) \at
              USTHB, Faculty of Mathematics, RECITS Laboratory P. Box $32$,\\ El-Alia, $16111$ Bab-Ezzouar, Algiers, Algeria.\\              \email{ydjemmada@usthb.dz or yahia.djem@gmail.com}
              \and
              S. HADJ-BRAHIM \at
              USTHB, Faculty of Mathematics, RECITS Laboratory P. Box $32$,\\ El-Alia, $16111$ Bab-Ezzouar, Algiers, Algeria.\\
              \email{bhadj@usthb.dz}
}

\date{Received: date / Accepted: date}

\maketitle

\begin{abstract}
The object of this paper is to introduce and study properties of unified Apostol-Bernoulli and Apostol-Euler polynomials noted by $\left\{\mathfrak{V_{n}}(x;\lambda;\mu)\right\}_{n \geq 0}$. We study some arithmetic properties of $\left\{\mathfrak{V_{n}}(x;\lambda;\mu)\right\}_{n \geq 0}$ as their connection to Apostol-Euler polynomials and Apostol-Bernoulli polynomials. Also, we give derivation and integration representations of $\left\{\mathfrak{V_{n}}(x;\lambda;\mu)\right\}_{n \geq 0}$. Finally, we use the umbral calculus approach to deduce symmetric identities.
\end{abstract}

\section{Introduction}
\label{intro}
The Bernoulli $\{B_{n}(x)\}_{n \geq 0}$ and the Euler $\{E_{n}(x)\}_{n \geq 0}$ polynomials respectively are generated by the following power series (see \cite{Bernoulli-1713, Euler-1738}):
\begin{equation*}
\dfrac{t}{e^{t}-1}e^{xt} =  \sum_{n = 0}^{\infty}B_{n}(x)\dfrac{t^{n}}{n!}  \qquad (\vert t\vert < 2\pi)
\end{equation*}
and
\begin{equation*}
\dfrac{2}{e^{t} + 1}e^{xt} =  \sum_{n = 0}^{\infty}E_{n}(x)\dfrac{t^{n}}{n!}  \qquad (\vert t\vert < \pi).
\end{equation*}
As a particular case, for $x=0$, we denote $B_{n} := B_{n}(0)$ and $E_{n} := E_{n}(0)$, which are called the Bernoulli and the Euler numbers respectively. They have numerous important applications in various fields of mathematics, as number theory, analysis and combinatorics.

\par Apostol \cite{Apostol1951} introduced and investigated the extended form of the classical Bernoulli polynomials and numbers, which are known as the Apostol-Bernoulli polynomials and numbers. The Apostol-Euler and the Apostol-Genocchi polynomials were introduced by Srivastava \cite{Srivastava2011}. Belbachir et \emph{al}. \cite{BH-2019, BHR-2019} proposed a new family of polynomials called Euler-Genocchi polynomials and studied their properties like linear recurrences and difference equations using a determinantal approach and generating function.
\section{Determinantal representation of the Bernoulli-Euler polynomials of Apostol type}
According to \cite{Luo-Srivastava}, the Apostol-Bernoulli polynomials $\{\mathfrak{B}_{n}(x;\lambda)\}_{n \geq 0}$ and the Apostol-Euler polynomials $\{\mathfrak{E}_{n}(x;\lambda)\}_{n \geq 0}$ are generated by the following power series:
\begin{equation}\label{Apostol-Bernoulli}
\dfrac{t}{\lambda e^{t}-1}e^{xt} =  \sum_{n = 0}^{\infty}\mathfrak{B}_{n}(x;\lambda)\dfrac{t^{n}}{n!} \qquad \left(\left|t+ \ln  \lambda\right|<2\pi, \ \lambda \in \mathbb{R}_{+}^{*}\right)
\end{equation}
and
\begin{equation}\label{Apostol-Euler}
\dfrac{2}{\lambda e^{t}+1}e^{xt} =  \sum_{n = 0}^{\infty}\mathfrak{E}_{n}(x;\lambda)\dfrac{t^{n}}{n!} \qquad \left(\left|t+ \ln \lambda\right|<\pi, \ \lambda \in \mathbb{R}_{+}^{*}\right),
\end{equation}
\noindent The Apostol-Bernoulli numbers $\mathfrak{B_{n}}(\lambda)$ and the Apostol-Euler numbers $\mathfrak{E_{n}}(\lambda)$ are given by
$\mathfrak{B_{n}}(\lambda) = \mathfrak{B}_{n}(0;\lambda)$ and $\mathfrak{E_{n}}(\lambda) = \mathfrak{E_{n}}(0;\lambda)$.

\vspace{0.3cm}
\noindent Letting
\begin{equation*}
T(x,\lambda,t)=\dfrac{2}{\lambda e^{t}+1}e^{xt}\times \dfrac{t}{\lambda e^{t}-1}e^{xt} = \dfrac{2t}{\lambda^{2} e^{2t}-1}e^{2xt}.
\end{equation*}
Taking into account the right hand side of \eqref{Apostol-Bernoulli} and \eqref{Apostol-Euler}, a direct computation gives
\begin{small}
\begin{equation*}
\lambda^{2} T(x+1,\lambda,t)-T(x,\lambda,t) = \sum_{n = 0}^{\infty}\left\{\sum_{k = 0}^{n}\binom{n}{k}\left[\lambda^{2}\mathfrak{B}_{n-k}(x,\lambda) \mathfrak{E}_{k}(x,\lambda) - \mathfrak{B}_{n-k}(x,\lambda) \mathfrak{E}_{k}(x,\lambda)\right]\right\}\dfrac{t^{n}}{n!}.
\end{equation*}
\end{small}

\noindent On the other hand, we have
\begin{small}
\begin{equation*}
\lambda ^{2} T(x+1,\lambda,t)-T(x,\lambda,t)  = \lambda ^{2} \dfrac{2t}{\lambda ^{2} e^{2t}-1}e^{2(x+1)t} - \dfrac{2t}{\lambda ^{2} e^{2t}-1}e^{2xt} = 2t e^{2xt} =  \sum_{n=0}^{+\infty}n2^{n}x^{n-1}\dfrac{t^{n}}{n!}.
\end{equation*}
\end{small}

\noindent Comparing the two expansions of $\lambda^{2} T(x+1,\lambda,t)-T(x,\lambda,t)$, we formulate the next result.
\begin{theorem}\label{determinantal-representation-Apostol-EB} Let ${x}$ be a real number and $n$ an integer. Then
\begin{small}
\begin{equation}\label{Apostol-monome-gen-EB-polnomials}
 x^{n} = \sum_{k = 0}^{n+1}\Lambda_{n,k} \times \Delta_{n+1-k,k}(x,\lambda),
\end{equation}
\end{small}
where \ $\Lambda_{n,k} = \dfrac{1}{2^{n + 1}(n + 1)}\binom{n+1}{k}$ \ and \ $\Delta_{n,k}(x,\lambda) = \left| \begin{array}{cc}
\lambda \mathfrak{B}_{n}(x+1,\lambda) & \hspace{-0.7cm} \mathfrak{E}_{k}(x,\lambda) \\
\hspace{-0.7cm} \mathfrak{B}_{n}(x,\lambda)     &  \lambda \mathfrak{E}_{k}(x+1,\lambda)
\end{array} \right|$.

\end{theorem}

In particular, taking $\lambda = 1$ in \eqref{Apostol-monome-gen-EB-polnomials}, we get the following result in terms of the Bernoulli and the Euler polynomials.
\begin{Cor}\emph{\cite{BHR-2019}}
Let ${x}$ be a real number and an integer $n\geq 0$, we have
\begin{equation}
\hspace{-1.8cm} x^{n} = \dfrac{1}{2^{n + 1}(n + 1)}\sum_{k = 0}^{n+1}\binom{n+1}{k}\left| \begin{array}{cc}
B_{n-(k-1)}(x + 1) & \hspace{-0.7cm} E_{k}(x) \\
\hspace{-0.7cm} B_{n-(k-1)}(x)     &  E_{k}(x + 1)
\end{array} \right|.
 \end{equation}
\end{Cor}
\section{Unified Bernoulli-Euler polynomials of Apostol type}

In this section, we give a definition of the unified Bernoulli-Euler polynomials of Apostol type and study their properties using power series.
\begin{defin}
Let $\lambda \in \mathbb{R}_{+}^{*}$ and $\mu \in \mathbb{R}_{+}-\{1\}$, we define the unified Bernoulli-Euler polynomials of Apostol type $\mathfrak{V_{n}}(x;\lambda;\mu)$ by the following power series:
\begin{equation}\label{EGP1}
\dfrac{2-\mu + \frac{\mu}{2}t}{\lambda e^{t}+(1-\mu)}e^{xt} =  \sum_{n \geq 0}\mathfrak{V_{n}}(x;\lambda;\mu)\dfrac{t^{n}}{n!},
\end{equation}
where
\begin{equation*}
\left\{
                       \begin{array}{ll}
                         \left|\ln\left(\frac{\lambda}{1-\mu}\right) + t \right|<2\pi, & \hbox{for } 0 \leq \mu < 1;\\
{}&{}\\
                         \left|\ln\left(\frac{\lambda}{\mu-1}\right) + t \right|<\pi, & \hbox{otherwise}.
                       \end{array}
                     \right.
\end{equation*}
Furthermore, the unified Bernoulli-Euler numbers of Apostol type, denoted $\mathfrak{V_{n}}(\lambda;\mu)$, are given by
\begin{equation}\label{EGP-1}
\mathfrak{V_{n}}(\lambda;\mu) := \mathfrak{V_{n}}(0;\lambda;\mu).
\end{equation}
\end{defin}

\noindent We summarize in the following table some special polynomials related to this extension.\\
\vspace{0.025cm}

\begin{scriptsize}
\begin{tabular}{|l|l|l|l|l|l|l|l|}
  \hline
  \textbf{Parameters} & \textbf{Generating functions} & \textbf{Polynomials} \\ \hline {}& {}&{}\\
   $\mu=0, \lambda=1$ &  $\displaystyle \dfrac{2}{e^{t} + 1}e^{xt} =  \sum_{n \geq 0}\limits E_{n}(x)\dfrac{t^{n}}{n!}, \qquad \left|t\right|<\pi$ &  Euler polynomials   \\
   {}&{}&{}\\
   \hline {}& {}&{}\\
   $\mu=2, \lambda=1$  &  $\displaystyle \dfrac{t}{e^{t} - 1}e^{xt} =  \sum_{n \geq 0}\limits B_{n}(x)\dfrac{t^{n}}{n!}, \qquad \left|t\right|<2\pi$ &   Bernoulli polynomials  \\
   {}&{}&{}\\
   \hline {}& {}&{}\\
$\mu=2$&$\displaystyle \dfrac{t}{\lambda e^{t}-1}e^{xt} =  \sum\limits_{n \geq 0}\mathfrak{B}_{n}(x;\lambda)\dfrac{t^{n}}{n!}, \qquad \left|t+ \ln  \lambda\right|<2\pi$ & Apostol-Bernoulli polynomials
   \\{}&{}& {}\\
  \hline {}& {}&{}\\
$\mu=0$ & $\displaystyle \dfrac{2}{\lambda e^{t}+1}e^{xt} =  \sum\limits_{n \geq 0}\mathfrak{E}_{n}(x;\lambda)\dfrac{t^{n}}{n!}, \qquad \left|t+ \ln \lambda\right|<\pi$
  & Apostol-Euler polynomials \\{}&{}& {}\\
  \hline
\end{tabular}
\end{scriptsize}

\vspace{1cm}
\noindent We list some properties of the unified Bernoulli-Euler polynomials of Apostol type using generating function approach.

\begin{theorem}\label{TH2}
Let $n$ be nonnegative integer, we have
\begin{equation}\label{EGP4}
\hspace{-2cm} \mathfrak{V}_{n}(x+y;\lambda;\mu) = \sum_{k = 0}^{n}\binom{n}{k}\mathfrak{V}_{k}(x;\lambda;\mu)y^{n-k}.
\end{equation}
\end{theorem}
\noindent In particular, for $x :=0$ and $y:=x$, the above relation becomes
\begin{equation}\label{expilcite-formula-representation}
\mathfrak{V}_{n}(x;\lambda;\mu) = \sum_{k = 0}^{n}\binom{n}{k}\mathfrak{V}_{k}(\lambda;\mu)x^{n-k}.
\end{equation}

\begin{proof}
We establish the  power series defined in \eqref{EGP1} for $\mathfrak{V}_{n}(x+y;\lambda;\mu)$, we have
\begin{equation*}
\sum_{n \geq 0}\mathfrak{V}_{n}(x+y;\lambda;\mu)\dfrac{t^{n}}{n!} = \left(\dfrac{2-\mu + \frac{\mu}{2}t}{\lambda e^{t}+(1-\mu)}\right)e^{(x+y)t} = \sum_{n \geq 0}\sum_{k \geq 0}\mathfrak{V}_{n}(x;\lambda;\mu){y}^{k}\dfrac{t^{n+k}}{n!k!}.
\end{equation*}
Applying the product series and then comparing the coefficients of  $t^{n}$ on both sides, we obtain Identity \eqref{EGP4}.
\end{proof}

\begin{rem}
Expression \eqref{EGP4} allows us to obtain $\mathfrak{V}_{n}(\lambda;\mu)$ the unified Bernoulli-Euler numbers of Apostol type in terms of the unified Bernoulli-Euler polynomials of Apostol type. Indeed, it suffices to replace $y$ by $-x$ in Formula \eqref{EGP4}, we get the following expression:
\begin{equation*}
\mathfrak{V}_{n}(\lambda;\mu) = \sum_{k = 0}^{n}\binom{n}{k}(-1)^{n-k}\mathfrak{V}_{k}(x;\lambda;\mu)x^{n-k}.
\end{equation*}
\end{rem}
As a first consequence of Theorem \ref{determinantal-representation-Apostol-EB}, we show that the unified Bernoulli-Euler polynomials of Apostol type, $\{\mathfrak{V}_{n}(x,\lambda;\mu)\}_{n \geq 0}$ given by the power series in \eqref{EGP1}, can be expressed in terms of the Apostol-Bernoulli and the Apostol-Euler polynomials. That is, by a straightforward calculation, the substitution of $x^{n}$ given by \eqref{Apostol-monome-gen-EB-polnomials} in Expression \eqref{expilcite-formula-representation} allows us to obtain the following formula:
\begin{prop}Let $n, k$ and $j$ be three integers, it holds that
\begin{equation*}
\mathfrak{V}_{n}(x;\lambda;\mu) = \sum_{k = 0}^{n}\binom{n}{k}\mathfrak{V}_{n-k}(\lambda;\mu)\sum_{j = 0}^{k+1}\Lambda_{k,j} \times \Delta_{k+1-j,j}(x,\lambda).
\end{equation*}
\end{prop}
\section{Generalized Raabe Theorem}
In this section, we give an extension of Raab Theorem for the unified Bernoulli-Euler polynomials of Apostol type.
\begin{theorem}\label{thm : Appell-polyn Euler-Genocchi type 2}
Let $r$ and $m$ be nonnegative integers with $m$ odd, for $\lambda = 1- \mu$ and $\mu \not = 1$, we have
\begin{equation*}
\sum\limits_{k=0}^{m-1}(-1)^{k}\mathfrak{V}_{n}\left(\frac{x+k}{m};1- \mu;\mu\right) =\frac{1-m}{m^{n}}\left(\dfrac{\mu-2}{2(\mu-1)}\right)E_{n}(x) + \frac{1}{m^{n-1}}\mathfrak{V}_{n}(x;1-\mu;\mu).
\end{equation*}
\end{theorem}
\begin{proof}\noindent  It follows from \eqref{EGP1} that
\begin{align*}
 \sum_{n=0}^{\infty} \sum\limits_{k=0}^{m-1}(-1)^{k}&\mathfrak{V}_{n}\left(\frac{x+k}{m};1- \mu;\mu\right) \frac{t^{n}}{n!}\\
                                               & = \sum_{k=0}^{m-1}(-1)^{k}\sum_{n=0}^{\infty}\mathfrak{V}_{n}\left(\frac{x+k}{m};1- \mu;\mu\right) \frac{t^{n}}{n!}\\
                                              & = \sum_{k=0}^{m-1}(-1)^{k}\dfrac{2-\mu+\frac{\mu}{2}t}{(1-\mu)e^{t}+(1-\mu)}e^{(x+k)t/m}\\
                                              & = \dfrac{2-\mu+\frac{\mu}{2}t}{(1-\mu)(e^{t}+1)}e^{xt/m}\dfrac{\left[1-(-e^{t/m})^{m}\right]}{1+e^{t/m}}\\
                                              & = \dfrac{1}{(1-\mu)}\left(\dfrac{2-\mu+\frac{\mu}{2}t}{e^{t/m}+1}\right)e^{xt/m}\\
                                              & = \dfrac{1-m}{2}\left(\dfrac{2-\mu}{1-\mu}\right)\left(\dfrac{2}{e^{t/m}+1}\right)e^{xt/m}  + \dfrac{m}{1-\mu}\left(\dfrac{2-\mu+\frac{\mu t}{2m}}{e^{t/m}+1}\right)e^{xt/m}\\
                                              & = \dfrac{1-m}{2}\left(\dfrac{2-\mu}{1-\mu}\right)\sum_{n=0}^{\infty}E_{n}(x)\dfrac{\left(t/m\right)^{n}}{n!} + \sum_{n=0}^{\infty}\frac{1}{m^{n-1}}\mathfrak{V}_{n}(x;1-\mu;\mu)\dfrac{t^{n}}{n!}.
\end{align*}
By simple manipulations and equating the coefficients of $t^{n}$ on both sides, we get the result.
\end{proof}
As a consequence of Theorem \ref{thm : Appell-polyn Euler-Genocchi type 2}, for $\mu = 2$ and $\mu = 0$ respectively, we have a multiplication Theorem for Euler and Bernoulli polynomials proved by Raabe in \cite{Raabe-1851}, as specified by Kargin and Kurt \cite{Kargin-Kurt-2011}. They are given as follows:
\begin{equation*}
\hspace{-0.7cm}\sum\limits_{k=0}^{m-1}B_{n}\left( \frac{x+k}{m}\right) =\dfrac{1}{m^{n-1}}B_{n}(x)
\end{equation*}
and
\begin{equation*}
\sum\limits_{k=0}^{m-1}(-1)^{k}E_{n}\left( \frac{x+k}{m} \right) =\frac{1}{m^{n}}E_{n}(x).
\end{equation*}
\section{Some explicit formulas}
In this section, we give some explicit formulas of the unified Bernoulli-Euler polynomials of Apostol type.
\begin{theorem}\label{Apostol-EB-polnomials}
For $\lambda \in \mathbb{R}_{+}^{*}$ and $\mu \in \mathbb{R}_{+} - \{1\}$, it holds that
\begin{small}
\begin{equation}\label{EGP18}
\mathfrak{V}_{n}(x;\lambda;\mu) = \frac{1}{2(\mu-1)}\left[(\mu-2)\mathfrak{E}_{n}\left(x;\frac{\lambda}{1-\mu}\right) -\frac{\mu n}{2}\mathfrak{E}_{n-1}\left(x;\frac{\lambda}{1-\mu}\right)\right]\quad (n \in \mathbb{N}).
\end{equation}
\end{small}
\end{theorem}

\begin{proof}
We can reformulate \eqref{EGP1} as follows
\begin{small}
\begin{align*}
\sum_{n = 0}^{\infty}\mathfrak{V_{n}}(x;\lambda;\mu)\dfrac{t^{n}}{n!}
& = \left(\frac{1}{2(\mu-1)}\right)\left((\mu-2) -  \frac{\mu}{2}t\right)\left(\frac{2}{1+\frac{\lambda}{1-\mu}e^{t}}\right)e^{xt}\\
& = \frac{1}{2(\mu-1)}\left[(\mu-2)\sum_{n = 0}^{\infty}\mathfrak{E}_{n}\left(x;\frac{\lambda}{1-\mu}\right)\dfrac{t^{n}}{n!} - \frac{\mu}{2}\sum_{n = 1}^{\infty}\mathfrak{E}_{n}\left(x;\frac{\lambda}{1-\mu}\right)\dfrac{t^{n+1}}{n!}\right]\\
& = \frac{1}{2(\mu-1)}\sum_{n = 1}^{\infty}\left[(\mu-2)\mathfrak{E}_{n}\left(x;\frac{\lambda}{1-\mu}\right) -\frac{\mu n}{2}\mathfrak{E}_{n-1}\left(x;\frac{\lambda}{1-\mu}\right)\right]\dfrac{t^{n}}{n!}.
\end{align*}
\end{small}
\noindent Equating the coefficients of $\frac{t^{n}}{n!}$ on both sides, we obtain Identity \eqref{EGP18}.
\end{proof}
Here, we give an explicit formula as a dual convex combination of classical Bernoulli and Euler polynomials of Apostol type.
\begin{theorem}
Let $n$ be nonnegative integer and a real number $\mu \not = 1$, we have
\begin{small}
\begin{equation}\label{EGP19}
\mathfrak{V}_{n}(x;\lambda;\mu) = \frac{1}{1-\mu}\left[\left(1-\frac{\mu}{2}\right)\mathfrak{E}_{n}\left(x;\frac{\lambda}{1-\mu}\right) - \frac{\mu}{2} \mathfrak{B}_{n}\left(x;\frac{\lambda}{\mu-1}\right)\right].
\end{equation}
\end{small}
\end{theorem}
\begin{proof}
From \eqref{EGP1}, we have
\begin{small}
\begin{equation*}
\sum_{n = 0}^{\infty}\mathfrak{V_{n}}(x;\lambda;\mu)\dfrac{t^{n}}{n!} = \dfrac{2- \mu + \frac{\mu}{2}t}{\lambda e^{t} + (1-\mu)}e^{xt} = \dfrac{2-\mu}{2(1-\mu)}\dfrac{2}{\frac{\lambda}{1-\mu} e^{t} + 1}e^{xt} + \dfrac{\mu}{2(\mu-1)}\dfrac{t}{\frac{\lambda}{\mu-1} e^{t} -1}e^{xt}.
\end{equation*}
\end{small}
Using \eqref{Apostol-Bernoulli} and \eqref{Apostol-Euler} leads to get \eqref{EGP19}.
\end{proof}
\begin{theorem}
For $\mu \not = 1$ and $n \geq 1$, the following formula holds:
\begin{small}
\begin{equation*}
(\mu-1)\sum_{k = 1}^{n}\binom{n}{k}\Upsilon_{n-k,k}\left(\frac{x}{2},\lambda,\mu\right) + \left(\frac{\mu}{2} -1\right)\Delta_{n,0}\left(\frac{x}{2},\frac{\lambda}{1- \mu}\right) = n(\mu-2)x^{n-1}-n(n-1)\mu x^{n-2},
\end{equation*}
\end{small}

\noindent where \ $\Upsilon_{n,k}(x,\lambda,\mu) = \left| \begin{array}{cc}
\frac{\lambda}{1-\mu} \mathfrak{B}_{n}(x+1,\lambda) & \hspace{-1.3cm} \mathfrak{V}_{k}(x;\lambda;\mu) \\
\hspace{-1.3cm} \mathfrak{B}_{n}(x,\lambda)     &  \frac{\lambda}{1-\mu} \mathfrak{V}_{k}(x+1;\lambda;\mu)
\end{array} \right|$.
\end{theorem}
\begin{proof}From Theorem \ref{Apostol-EB-polnomials}, we have
\begin{small}
\begin{align*}
& \sum_{k = 1}^{n}\binom{n}{k}\Upsilon_{n-k,k}\left(x,\lambda,\mu\right) \\& = \frac{1}{2(\mu-1)}\left\{(\mu-2)\sum_{k = 1}^{n}\binom{n}{k}\Delta_{n-k,k}\left(x,\frac{\lambda}{1- \mu}\right) -\mu \sum_{k = 1}^{n}k\binom{n}{k}\Delta_{n-k,k-1}\left(x,\frac{\lambda}{1- \mu}\right)\right\}\\
& = \frac{1}{2(\mu-1)}\left\{(\mu-2)\sum_{k = 1}^{n}\binom{n}{k}\Delta_{n-k,k}\left(x,\frac{\lambda}{1- \mu}\right) -\mu n \sum_{k = 1}^{n}\binom{n-1}{k-1}\Delta_{n-k,k-1}\left(x,\frac{\lambda}{1- \mu}\right)\right\}.
\end{align*}
\end{small}
Applying Theorem  \ref{determinantal-representation-Apostol-EB} and a straightforward computation, we obtain

\begin{align*}
\sum_{k = 1}^{n}\binom{n}{k}\Upsilon_{n-k,k}\left(x,\lambda,\mu\right) &= \frac{1}{\mu - 1}\left\{n(\mu-2)(2x)^{n-1} - n(n-1)\mu (2x)^{n-2}\right\} \\
                                                                       &- \frac{\mu - 2}{2(\mu - 1)}\Delta_{n,0}\left(x,\frac{\lambda}{1- \mu}\right).
\end{align*}
Multiplying both sides by $(\mu - 1)$, we get the desired identity.
\end{proof}
\section{Derivation and integration representations of unified Bernoulli-Euler polynomials of Apostol type}
In this section, we present derivation and integration representations for the unified Bernoulli-Euler polynomials of Apostol type.
\begin{theorem}\label{TH4}
Let $l, n$ be two nonnegative integers. Then
\begin{equation}\label{EGP7}
\hspace{-3.5cm} \dfrac{d^{l}}{dx^{l}}\mathfrak{V}_{n}(x;\lambda;\mu) = (n)_{l}\mathfrak{V}_{n-l}(x;\lambda;\mu),
\end{equation}
\begin{equation}\label{EGP8}
\int_{x}^{y}\mathfrak{V}_{n}(z;\lambda;\mu)dz  = \dfrac{1}{(n+1)}\left[\mathfrak{V}_{n+1}(y;\lambda;\mu) - \mathfrak{V}_{n+1}(x;\lambda;\mu)\right],
\end{equation}
where \quad  $(x)_{n} := x(x-1)\cdots(x-n+1)$ \quad with \quad $(x)_{0} = 1$.
\end{theorem}

\begin{proof}
The assertion \eqref{EGP7} follows from \eqref{EGP1} by successive differentiations with respect to $x$ and then uses the induction principle on $l$. Furthermore, taking $l = 1$ in (\ref{EGP7}) and integrating both sides of the resulting equation with respect to $z$ over the interval $[x, y], (y > x)$, we obtain the Integral Formula (\ref{EGP8}).
\end{proof}
\begin{rem}
Setting $\lambda = 1, \mu = 2$ in \eqref{EGP7} and \eqref{EGP8}, we obtain known results due to Luo et \emph{al}. \emph{\cite{Luo-Srivastava}}.
\end{rem}

\begin{Cor}
Let $n$ be a nonnegative integer. Then

\begin{equation}\label{Int1}
\int_{x}^{x+y}\mathfrak{V}_{n}(z;\lambda;\mu)dz = \dfrac{1}{(n+1)}\sum_{k = 0}^{n}\binom{n}{k}\mathfrak{V}_{k}(x;\lambda;\mu)y^{n-(k-1)}.
\end{equation}
\end{Cor}

\begin{proof}
Replacing $y$ by $x+y$ in the Integral Formula \eqref{EGP8} and using Formula \eqref{EGP4}, by successive calculations, we obtain the Integral Formula \eqref{Int1}.
\end{proof}

\begin{theorem}\label{TH3}
For $\mu \in \mathbb{R}_{+}^{*} - \{1,2\}$ and $n$ nonnegative integer, the following formula holds:
\begin{align}\label{EGP11}
& \mathfrak{V}_{n+1}(x;\lambda;\mu) - x\mathfrak{V}_{n}(x;\lambda;\mu) =  \dfrac{1}{2-\mu}\sum_{i = 0}^{n}\binom{n}{i}(n-i)! \nonumber \\
& \times \left(\dfrac{\mu}{2(\mu-2)}\right)^{n-i} \left[ \frac{\mu}{2}\mathfrak{V}_{i}(x;\lambda;\mu) - \lambda \sum_{k=0}^{i}\binom{i}{k}\mathfrak{V}_{k}(\lambda;\mu)\mathfrak{V}_{i-k}(x+1;\lambda;\mu)\right].
\end{align}
\end{theorem}

\begin{proof}
Differentiating both sides of \eqref{EGP1} with respect to $t$, we express the factors $\left(1+ \frac{\mu}{2(2-\mu)}t\right)^{-1}$ in series form for $\left|t\right|<\frac{2}{\mu}\left|2-\mu\right|$, and using Formulas \eqref{EGP1} and \eqref{EGP-1}, we obtain
\begin{small}
\begin{align}\label{EGP10}
& \sum_{n = 0}^{\infty}\mathfrak{V}_{n+1}(x;\lambda;\mu)\dfrac{t^{n}}{n!} =x\sum_{n = 0}^{\infty}\mathfrak{V}_{n}(x;\lambda;\mu)\dfrac{t^{n}}{n!}+ \dfrac{1}{(2-\mu)}\left(\sum_{n = 0}^{\infty}\frac{\mu^{n}}{2^{n}(\mu-2)^{n}}t^{n}\right)\nonumber \\
& \times \left[\frac{\mu}{2}\sum_{n = 0}^{\infty}\mathfrak{V}_{n}(x;\lambda;\mu)\dfrac{t^{n}}{n!} - \lambda \left(\sum_{n = 0}^{\infty}\mathfrak{V}_{n}(\lambda;\mu)\dfrac{t^{n}}{n!}\right)\left(\sum_{n = 0}^{\infty}\mathfrak{V}_{n}(x+1;\lambda;\mu)\dfrac{t^{n}}{n!}\right)\right].
\end{align}
\end{small}
Then taking into account the series product in \eqref{EGP10} and equating the coefficients of $t$, we get Identity \eqref{EGP11}.
\end{proof}
\section{Identities inspired via umbral calculus}
Umbral calculus approach is a usefull tool to get and guess arithmetic and combinatorial identities, see for instense Gessel \cite{gess} on some applications of the classical umbral calculus, Di Crescenzo et \emph{al}. \cite{cre} on umbral calculus. See also classical references as those of Roman and Rota \cite{Roman-Rota-1978}.\\
Let $\mathbf{B}^{n}(\lambda;\mu)$ be the umbra defined by \ $\mathbf{B}^{n}(\lambda;\mu):=\mathfrak{V}_{n}\left(\lambda;\mu\right)$ and $\left(
\mathfrak{V}_{n}\left( x;\lambda;\mu\right)\right) _{n\geq 0}\ $ defined by %
\begin{equation*}
\sum_{n\geq 0}\mathfrak{V}_{n}\left( x;\lambda;\mu\right) \frac{t^{n}}{n!}=F\left( t\right)e^{xt} =\exp \left( \left(\mathbf{B}(\lambda;\mu)+x\right) t\right),
\end{equation*}%
where $F\left( t\right) :=\sum\limits_{n\geq 0}\mathfrak{V}_{n}\left(\lambda;\mu\right)\frac{t^{n}}{n!}=\exp \left( \mathbf{B}(\lambda;\mu) t\right).$
So,  $\mathfrak{V}_{n}\left( x;\lambda;\mu\right)$ admits the umbral representation
\begin{equation*}
\mathfrak{V}_{n}\left( x;\lambda;\mu\right) =\left( \mathbf{B}(\lambda;\mu)+x\right) ^{n}.
\end{equation*}%

\begin{theorem}
Let $n$ be nonnegative integer. Then
\begin{equation*}\label{EGP2}
\mathfrak{V}_{n}(x+1;\lambda;\mu) = \sum_{k = 0}^{n}\binom{n}{k}\mathfrak{V}_{k}(x;\lambda;\mu).
\end{equation*}
\end{theorem}

\begin{proof}
By the umbral representation $\mathfrak{V}_{n}\left( x;\lambda;\mu\right) =\left( \mathbf{B}(\lambda;\mu)+x\right) ^{n}$, we have
\begin{small}
\begin{equation*}
\mathfrak{V}_{n}(x+1;\lambda;\mu) = \left( \mathbf{B}(\lambda;\mu)+(x+1)\right) ^{n}
                                   = \sum_{k = 0}^{n}\binom{n}{k}\left( \mathbf{B}(\lambda;\mu)+x\right) ^{k} = \sum_{k = 0}^{n}\binom{n}{k}\mathfrak{V}_{k}(x;\lambda;\mu).
\end{equation*}
\end{small}
\end{proof}

\begin{theorem} Let $n,m$ be nonnegative integers. Then%
\begin{small}
\begin{equation*}
\sum_{k=0}^{n}\binom{n}{k}y^{n-k} \mathfrak{V}_{m+k}\left( x;\lambda;\mu\right) =\sum_{k=0}^{m}%
\binom{m}{k}\left( -y\right) ^{m-k}\mathfrak{V}_{n+k}\left( x+y;\lambda;\mu\right).
\end{equation*}
\end{small}
\end{theorem}

\begin{proof}
By the umbral representation $\mathfrak{V}_{n}\left( x;\lambda;\mu\right) =\left( \mathbf{B}(\lambda;\mu)+x\right) ^{n}$,
on the one hand, we have
\begin{small}
\begin{align*}
(\mathbf{B}(\lambda;\mu)+(x+y)) ^{n}(\mathbf{B}(\lambda;\mu)+x)^{m} & = \left( \mathbf{B}(\lambda;\mu)+(x+y) \right) ^{n}\left( \mathbf{B}(\lambda;\mu)+(x+y)-y\right) ^{m} \\
& = \sum_{k=0}^{m}\binom{m}{k}\left( -y\right)^{m-k}\left( \mathbf{B}(\lambda;\mu)%
+(x+y)\right) ^{n+k} \\
& = \sum_{k=0}^{m}\binom{m}{k}\left( -y\right)^{m-k}\mathfrak{V}_{n+k}\left( x+y;\lambda;\mu\right),
\end{align*}
\end{small}
and on the other hand, we have
\begin{small}
\begin{align*}
(\mathbf{B}(\lambda;\mu)+(x+y))^{n}(\mathbf{B}(\lambda;\mu)+x)^{m} & = \sum_{k=0}^{n}\binom{n}{k}y^{n-k}\left( \mathbf{B}(\lambda;\mu)+x\right) ^{m+k} \\
& = \sum_{k=0}^{n}\binom{n}{k}y^{n-k}\mathfrak{V}_{m+k}(x;\lambda;\mu).
\end{align*}
\end{small}
Hence, the two expressions give the desired identity.
\end{proof}

\bibliography{Ref.bib}
\bibliographystyle{acm}
\end{document}